\documentclass{amsart}
\usepackage{amsmath, amsfonts, latexsym, array, amssymb, amscd }
\usepackage[all]{xypic}
\usepackage[all]{xy}
\usepackage{graphicx}
\usepackage{color}
\usepackage{enumitem}
\usepackage{epstopdf}
\usepackage{hyperref}

\numberwithin{equation}{section}
\numberwithin{figure}{section}

\newtheorem{thm}{Theorem}[section]

\newtheorem{cor}[thm]{Corollary}

\newtheorem{prop}[thm]{Proposition}

\newtheorem{defn}[thm]{Definition}
\newtheorem{rem}[thm]{Remark}

\newtheorem{thmx}{Theorem}

\newcommand{\comment}[1]{}

\usepackage{ulem}

\newcommand{\ZZ}{\mathbb{Z}}
\newcommand{\RR}{\mathbb{R}}

\newcommand{\ph}{\varphi}
\newcommand{\PPP}{\mathcal{P}}
\newcommand{\SSS}{\mathcal{S}}
\newcommand{\GGG}{\mathcal{G}}

\newcommand{\CCC}{\mathcal{C}}

\newcommand{\MMM}{\mathcal{M}}

\newcommand{\eps}{\epsilon}

\newcommand{\FFF}{\mathcal{F}}

\DeclareMathOperator{\NE}{NE}

\usepackage{mathrsfs}

\newcommand{\x}{\times}
\newcommand{\E}{\mathbb{E}}
\newcommand{\R}{\RR}
\newcommand{\N}{\mathbb{N}}
\newcommand{\Q}{\mathbb{Q}}

\newcommand{\s}[1]{\mathscr{#1}}

\providecommand{\abs}[1]{\left | #1 \right |}
\providecommand{\norm}[1]{\lVert #1 \rVert}

\newcommand{\tGGG}{\tilde{\mathcal{G}}}

\begin{document}
\title[The K-Property for Unique Equilibrium States]{The K-Property for Some Unique Equilibrium States in Flows and Homeomorphisms}
\author{Benjamin Call}
\address{B. Call, Department of Mathematics, The Ohio State University, Columbus, OH 43210, \emph{E-mail address:} \tt{call.119@buckeyemail.osu.edu}}

\date{\today}
\thanks{B.C.\ is partially supported by NSF grant DMS-$1461163$.}
\commby{}

\begin{abstract}
We set out some general criteria to prove the $K$-property, refining the assumptions used in \cite{CaT} for the flow case, and introducing the analogous discrete-time result. We also introduce one-sided $\lambda$-decompositions, as well as multiple techniques for checking the pressure gap required to show the $K$-property. We apply our results to the family of Ma\~n\'e diffeomorphisms and the Katok map. Our argument builds on the orbit decomposition theory of Climenhaga and Thompson.
\end{abstract}

\maketitle
\setcounter{tocdepth}{1}

\section{Introduction}

Given a dynamical system $(X,f)$ and a continuous potential $\ph : X\to\R$, we call any invariant measure $\mu$ such that the measure-theoretic pressure $P_\mu(\ph)$ is equal to the topological pressure $P(\ph)$ an \emph{equilibrium state}, where we write $P_\mu(\ph) = h_\mu(f) + \int\ph\,d\mu$. The existence of equilibrium states is guaranteed by upper semicontinuity of the entropy map $\mu\mapsto h_\mu(f)$, as can be seen from the Variational Principle; see \cite[Chapter 9.3]{Wa} for more details.

\begin{prop}\label{prop: var prin}
Let $X$ be a compact metric space, $f : X\to X$ a homeomorphism, and $\ph : X\to\R$ continuous. Then
$$P(\ph) = \sup \left\{h_\mu(\ph) + \int\ph\,d\mu\mid \mu\in\MMM(X,f)\right\}$$
\end{prop}

Uniqueness is a more difficult question, and conditions that imply it have been much studied. In classical settings, such as when $f$ is Anosov and $\ph$ is H\"older continuous, equilibrium states are unique and have strong mixing and statistical properties. Proofs of this have been obtained using a variety of different techniques. One method, due to Bowen \cite{Bow76}, provides three conditions which need to be checked to guarantee uniqueness of these equilibrium states (and, as later work would show \cite{L}, strong mixing properties). Later, Climenhaga and Thompson \cite{CT} introduced the theory of orbit decompositions, weakening each of Bowen's conditions to ``non-uniform'' versions, and this theory has been applied successfully in a variety of settings \cite{CFT,CFT2,Wa19,CKP,BCFT}. However, this theory does not a priori provide any mixing properties of the unique equilibrium states.

The $K$-property is a mixing property which is stronger than mixing of all orders, and weaker than Bernoulli. Finding examples of systems with the $K$-property which are not Bernoulli, as well as cases when $K$ can be shown to imply Bernoulli is an active area of research \cite{KHV,PTV}. In \cite{CaT}, the author and Thompson adapt Ledrappier's criterion for proving the $K$-property \cite{L} to the flow setting in order to show that some systems with unique equilibrium states built through the Climenhaga-Thompson decomposition have the $K$-property. The motivation behind this was to apply it to geodesic flows on rank 1 manifolds of non-positive curvature.

In this paper, we will refine and remove many of the assumptions used in \cite{CaT} to show the $K$-property, proving the following.

\begin{thmx}\label{Theorem A}
	Let $(X,\FFF)$ be a continuous flow on a compact metric space, and $\ph : X\to \R$ a continuous potential. Suppose that $P_{\exp}^\perp(\ph) < P(\ph)$ and that $(X,\FFF)$ is asymptotically entropy expansive. Then suppose that $X\x [0,\infty)$ has a $\lambda$-decomposition $(\PPP,\GGG,\SSS)$ with the following properties:
	\begin{enumerate}
		\item $\GGG(\eta)$ has specification at every scale $\delta > 0$ for all $\eta > 0$;
		\item $\ph$ has the Bowen property on $\GGG(\eta)$ for all $\eta > 0$;
	\end{enumerate}
	and furthermore that $P(\bigcap_{t\in\R}(f_t\x f_t)\tilde{\lambda}^{-1}(0),\Phi) < 2 P(\ph)$, where $\Phi(x,y) = \ph(x) + \ph(y)$ and $\tilde{\lambda}(x,y) = \lambda(x)\lambda(y)$.
	Then $(X,\FFF,\ph)$ has a unique equilibrium state, and it is $K$.
\end{thmx}

\begin{rem}
	This theorem uses slightly stronger assumptions than we will use in the proof. In particular, the assumption of asymptotic entropy expansivity can be replaced by the condition that the entropy map on the product space $(X\x X,\FFF\x \FFF)$ is upper semicontinuous.
\end{rem}

There are many techniques for showing uniqueness of equilibrium states outside of the uniformly hyperbolic setting, each with different advantages, both in terms of ease of application and strength of results. See \cite{ClimPes} for a thorough review of these techniques in the non-uniformly hyperbolic setting. Theorems \ref{Theorem A} and \ref{Theorem B} provide mild conditions under which the orbit decomposition theory of \cite{CT} gives the $K$-property.

The theory of $\lambda$-decompositions can be translated to the discrete-time setting, and the corresponding theorem holds as well, with the proof simplifying somewhat.

\begin{thmx}\label{Theorem B}
Let $(X,f)$ be a homeomorphism on a compact metrix space and $\ph : X\to \R$ a continuous potential. Suppose that $P_{\exp}^\perp(\ph) < P(\ph)$ and $(X,f)$ is asymptotically entropy expansive. Suppose $X\x \N$ has a $\lambda$-decomposition $(\PPP,\GGG,\SSS)$ with the following properties:
\begin{enumerate}
	\item $\GGG(\eta)$ has specification at every scale $\delta > 0$ for all $\eta > 0$;
	\item $\ph$ has the Bowen property on $\GGG(\eta)$ for all $\eta > 0$;
\end{enumerate}
and furthermore,
$$P\left(\bigcap_{n\in\ZZ} (f^n\x f^n)\tilde{\lambda}^{-1}(0),\Phi\right) < 2P(\ph).$$
Then $(X,f,\ph)$ has a unique equilibrium state, and it has the $K$-property.
\end{thmx}
In some applications of \cite{CT}, either the collection of prefixes or the collection of suffixes is empty. However, $\lambda$-decompositions cannot apply to any such example. Consequently, we introduce \emph{one-sided $\lambda$-decompositions} to account for these examples, and show that the analogous results for the $K$-property hold as well. In particular, these decompositions were used in \cite{CFT2} to establish unique equilibrium states for small $C^1$-perturbations of a class of Ma\~n\'e diffeomorphisms.

The product pressure gap in both the discrete-time and flow cases is non-trivial to check, even with the assumption of a pressure gap in the base system. We provide two different methods which can be used to check this gap for future applications of these results. We then show that they can be applied to equilibrium states for the Katok map and Ma\~n\'e diffeomorphisms for some potentials sufficiently close to constant. Although this result is not new for the Katok map, it outlines a relatively simple approach to obtaining the $K$-property in this setting. However, to the best of the author's knowledge, this is a new result for non-constant potentials for Ma\~n\'e diffeomorphisms, and we obtain the following result.
\begin{thmx}\label{Theorem C}
	For any Ma\~n\'e diffeomorphism $f$ sufficiently $C^0$-close to Anosov, and for a $C^1$-open set of perturbations of $f$, the unique equilibrium state for all H\"older continuous potentials sufficiently close to a constant is $K$.
\end{thmx}

We note that this applies immediately to potentials $\ph = t\ph^u$ for $t$ close to $0$, where $\ph^u$ is the geometric potential and to the SRB measure when $\ph^u$ is close to constant.

The structure of this paper is as follows. In \S \ref{background}, we present the necessary background. In \S \ref{Expansivity} and \S \ref{Pressure}, we prove Theorem \ref{Theorem A}. We also provide two methods for checking the pressure estimates that appear in Theorems \ref{Theorem A} and \ref{Theorem B}. In \S \ref{Discrete}, we prove Theorem \ref{Theorem B} and introduce one-sided $\lambda$ decompositions. Finally, in \S \ref{Applications} we discuss applications of these results to the Katok map and Ma\~n\'e diffeomorphisms, proving Theorem \ref{Theorem C}.

\section{Background} \label{background}
Throughout this section, let $X$ be a compact metric space, $\FFF = (f_t)_{t\in\R}$ a continuous flow on $X$, $f : X\to X$ a homeomorphism, and $\ph : X\to\R$ a continuous potential. Finally, $\MMM(X,f)$ will denote the space of $f$-invariant probability measures on $X$.

\subsection{The K-Property}
The $K$-property is a mixing property that is stronger than mixing of all orders and weaker than Bernoulli. It was originally defined by the existence of an invariant algebra with certain properties.

\begin{defn}
$(X,f,\s{B},\mu)$ has the $K$-property if there is a sub $\sigma$-algebra $\mathscr{K}\subset\s{B}$ such that $f\s{K}\supset \s{K}$, $\bigvee_{i=0}^\infty f^i\s{K} = \s{B}$ and $\bigcap_{i=0}^\infty f^i\s{K} = \{0,X\}$.
\end{defn}

Lifting this definition to flows follows in an elegant manner:

\begin{defn}
$(X,\FFF,\mu)$ is a $K$-flow if and only if for some $t\neq 0$, $(X,f_t,\mu)$ is $K$.
\end{defn}

There are multiple equivalent definitions of the $K$-property all of which hold for flows as well by the above definition. We collect some of them here, omitting the proofs, and instead referring the reader to \cite{CFSS}.

\begin{prop}
A dynamical system $(X,f,\mu)$ has the $K$-property if and only if it has completely positive entropy---that is, there exists no non-trivial factor with zero entropy.
\end{prop}

Equivalently, defining the \emph{Pinsker factor} to be the largest zero entropy factor, this says that $(X,f,\mu)$ has the $K$-property if and only if the Pinsker factor is trivial. Formulating the $K$-property in the form of mixing conditions yields

\begin{prop}
$(X,f,\mu)$ is $K$ if given finitely many measurable sets $A_1,\cdots, A_n$ and $B$,
$$\lim\limits_{m\to\infty}\sup_{C\in \s{A}_m}\abs{\mu(B\cap C) - \mu(B)\mu(C)} = 0$$
where $\s{A}_m$ is the minimal $\sigma$-algebra generated by $\{f^kA_i\mid 1\leq i\leq n, k\geq m\}$.
\end{prop}

\begin{rem}
Although we do not consider the Bernoulli property in this work, we mention this result, as many proofs showing that systems with the $K$-property imply the Bernoulli property make use of this condition. See for instance \cite{CH96, OW98, Rat, LLS}.
\end{rem}

Beyond these equivalent definitions, the $K$-property also implies mixing of all orders, which can be seen from $K$-mixing, as well as continuous Lebesgue spectrum, which can be seen from the original definition of the $K$-property.

To show the $K$-property, we will make use of the following result of Ledrappier.

\begin{thm}[{\cite[Proposition 1.4]{L}}]\label{Ledr}
	Let $(X,f)$ be a weakly expansive (or asymptotically $h$-expansive \cite{Led78}) system, and let $\ph$ be a continuous function on $X$. Let $(X\x X,f\x f)$ be the product of two copies of $(X,f)$ and $\Phi(x_1,x_2) := \ph(x_1) + \ph(x_2)$. If $\Phi$ has a unique equilibrium measure in $\MMM(X\x X,f\x f)$, then the unique equilibrium measure for $\ph$ in $\MMM(X,f)$ has the Kolmogorov property.
\end{thm}

\begin{rem}
In \cite[Proposition 2.7]{CaT}, the author and Thompson show that this result holds for flows as well.
\end{rem}

In \S \ref{Expansivity}, we will show that the weak expansivity condition can be removed from the statement, without changing the result.

Finally, we note that unlike other mixing conditions, the $K$-property is necessarily limited to invertible systems. There is a one-sided analogue, called \emph{exact}, in that the natural extension of any exact system has the $K$-property (see \cite{Rokh}). As the results of \cite{CT} are not limited to invertible systems, it would be interesting to study this further.

\subsection{Fiber Entropy and Joinings}
We introduce material used in \S \ref{Expansivity}, specifically fiber entropy, as well as joinings and disintegrations of measures. For more details and background, we refer the reader to the books \cite{downarowicz,Furstenberg}.

\subsubsection{Fiber Entropy}
We briefly recall the standard definitions of measure-theoretic entropy. The entropy of a finite partition $\xi$ with respect to a probability measure $\mu$ is defined as
$$H_\mu(\xi) := -\sum_{A\in \xi} \mu(A)\log(\mu(A))$$
and the conditional entropy of $\xi$ with respect to another finite partition $\eta$ is
$$H_\mu(\xi\mid \eta) := -\sum_{A\in\xi,B\in\eta} \mu(A\cap B)\log\left(\frac{\mu(A\cap B)}{\mu(B)}\right).$$
Then, introducing the dynamics, we define the entropy of a finite partition with respect to a transformation as
$$h_\mu(\xi;f) := \lim\limits_{n\to\infty}H_\mu\left(\xi\mid \bigvee_{i=0}^{n-1}f^{-i}\xi\right).$$
Finally, the entropy of a transformation is
$$h_\mu(f) := \sup_{\xi} h_\mu(\xi;f).$$

Computing this can be non-trivial in many cases, and for our proof of Proposition \ref{Ledr}, we will make use of a concept called fiber entropy, which was first introduced by Abramov and Rokhlin in \cite{AbrRo}, and later generalized by Downarowicz and Serafin in \cite{DownSer}. We first recall the definition of the disintegration of a measure over a factor.

\begin{defn}
	Let $(Y,\s{A},g,\nu)$ be a factor of $(X,\s{B},f,\mu)$ via the factor map $\pi : X\to Y$. The disintegration of $\mu$ over $Y$ is a collection of probability measures $\{\mu_y\}_{y\in Y}$ such that $\mu_y$ is fully supported on $\pi^{-1}(y)$ for $\nu$-a.e. $y\in Y$, and for all $A\in\s{B}$, 
	$$\mu(A) = \int_Y \mu_y(A)\,d\nu(y).$$
\end{defn}

\begin{rem}
Any two disintegrations agree for $\nu$-a.e. $y\in Y$. Therefore, defining the disintegration of a measure over a factor is equivalent to defining the measure.
\end{rem}

\begin{defn}
Let $(Y,g,\nu)$ be a factor of $(X,f,\mu)$, and let $\xi$ be a partition of $X$ with finite entropy. Then for $\nu$-a.e. $y\in Y$, the \emph{fiber entropy} of $\xi$ with respect to $\nu$ is defined as
$$H_\mu(\xi\mid y) := \lim\limits_{n\to\infty}H_{\mu_y}\left(\xi\mid \bigvee_{i=0}^{n-1}f^{-i}\xi\right),$$
where $\{\mu_y\}$ is the disintegration of $\mu$ over $Y$.
\end{defn}

Using this, Downarowicz and Serafin generalized the Abramov-Rokhlin formula to hold for individual partitions.
\begin{prop}[{\cite[Theorem 1]{DownSer}}]\label{thm: fiber entropy}
If $(Y,g,\nu)$ is an invertible factor of $(X,f,\mu)$, then for every finite partition $\xi$ with finite entropy,
$$h_\mu(\xi\mid \pi^{-1}(\s{A})) = \int H_\mu(\xi\mid y)\,d\nu.$$
\end{prop}
Taking the supremum over all $\xi$, the classical Abramov-Rokhlin formula is obtained.

\subsubsection{Joinings}
\begin{defn}
	Given two dynamical systems $(X,\s{A},f,\mu)$ and $(Y,\s{B},g,\nu)$, a probability measure $\lambda \in \MMM(X\x Y, f\times g)$ is a \emph{joining} of $\mu$ and $\nu$ if the marginals of $\lambda$ are $\mu$ and $\nu$ respectively. In other words, for any $A\in\s{A}$ and $B\in\s{B}$, $\lambda(A\x Y) = \mu(A)$ and $\lambda(X\x B) = \nu(B)$.
\end{defn}

For any two dynamical systems, there is always the \emph{independent joining}, which is just the product measure. When $X$ and $Y$ share a common factor, we can define the \emph{relatively independent joining} over this factor.

\begin{defn}
	Suppose that $(Z, \s{C},h,\rho)$ is a factor of both $(X,\s{A},f,\mu)$ and $(Y,\s{B},g,\nu)$. Let $\{\mu_z\}$ and $\{\nu_z\}$ be the disintegrations of $\mu$ and $\nu$ over $Z$. Then the relatively independent joining is defined by
	$$\mu \otimes_Z \nu (A) := \int_Z \mu_z\x\nu_z(A) \,d\rho(z).$$
\end{defn}

We will use one particularly nice feature of joinings in this paper, specifically, that entropy cannot increase beyond the sum of the marginals. The proof is an easy exercise.

\begin{prop}[{\cite[Fact 4.4.3]{downarowicz}}]\label{prop: joining entropy}
Let $\lambda$ be a joining of $\mu$ and $\nu$. Then
$$h_\lambda(f\x g) \leq h_\mu(f) + h_\nu(g).$$
\end{prop}

This result can be shown to propagate to pressure as well, showing that if a joining is an equilibrium state for an appropriate potential, it must be a joining of equilibrium states. While easy to prove, this is key to multiple results throughout this paper.
\begin{prop}\label{prop: projection of ES}
	Let $(X_i,f_i,\mu_i)$ be two dynamical systems, and $\ph_i$ continuous potentials on $X_i$. Then, set $\Phi(x,y) = \ph_1(x) + \ph_2(y)$. For any joining $\lambda$ of $\mu_1$ and $\mu_2$,
	$$P_\lambda(\Phi) \leq P_{\mu_1}(\ph_1) + P_{\mu_2}(\ph_2).$$
	In particular, if $\lambda$ is an equilibrium state, then $\mu_i$ is an equilibrium state for $(X_i,f_i,\ph_i)$.
\end{prop}

\begin{proof}
By Proposition \ref{prop: joining entropy}, $h_\lambda(f_1\x f_2) \leq h_{\mu_1}(f_1) + h_{\mu_2}(f_2)$. As $\Phi$ acts independently on each coordinate, we have that $\int\Phi\,d\mu = \int\ph_1\,d\mu_1 + \int\ph_2\,d\mu_2$, completing the proof of the inequality. If $\lambda$ is an equilibrium state for $\Phi$, then
$$P(\ph_1) + P(\ph_2) = P(\Phi) = P_\lambda(\Phi) \leq P_{\mu_1}(\ph_1) + P_{\mu_2}(\ph_2) \leq P(\ph_1) + P(\ph_2),$$
where the first equality can be found in \cite[Theorem 9.8]{Wa} and the last inequality is by the Variational Principle. Consequently, we see that $\mu_1$ and $\mu_2$ are both equilibrium states.
\end{proof}

\subsection{Orbit Decompositions}

In \cite{CT}, Climenhaga and Thompson introduced the idea of decomposing orbit segments as a tool to show uniqueness of equilibrium states. We represent the space of orbit segments as $X\x [0,\infty)$, with $(x,t)$ identified with the orbit segment $\{f_sx\mid s\in [0,t]\}$. In this section, we recall the formal definition of an orbit decomposition, as well as the ways that pressure, specification, and the Bowen property are adapted to this setting.

\begin{defn}
	Let $\FFF$ be a continuous flow on a compact metric space $X$. A decomposition of $X\x [0,\infty)$ is a trio $(\PPP,\GGG,\SSS)$ such that there exist functions $p,g,s : X\x [0,\infty)\to [0,\infty)$ so that for any orbit segment $(x,t)$, writing $p := p((x,t))$ and similarly for $g$ and $s$, we have $t = p+ g + s$ and
	$$(x,p)\in\PPP,\quad (f_px,g)\in\GGG,\quad (f_{p+g}x,s)\in\SSS.$$
\end{defn}
Informally, this should be thought of breaking down an orbit segment into three parts, consisting of a ``prefix'', ``suffix'', and ``good'' part, where the collection of bad orbit segments (the prefix and the suffix) have less pressure than the full space, and the dynamics and potential exhibit good behavior on $\GGG$. In \cite{CaT}, a specific class of decompositions was introduced, termed $\lambda$-decompositions, and this is the class that we will study in this paper.

\begin{defn}
	Let $\lambda : X\to [0,\infty)$ be bounded and lower semicontinuous. Then, for all $\eta > 0$, define
	$$B(\eta) := \left\{(x,t)\mid \frac{1}{t}\int_{0}^{t}\lambda(f_sx)\,ds < \eta\right\}$$
	to be the class of ``bad'' orbit segments at scale $\eta$ and
	$$\GGG(\eta) := \left\{(x,t)\mid \frac{1}{r}\int_{0}^{r}\lambda(f_sx)\,ds \geq \eta \text{ and } \frac{1}{r}\int_{t-r}^{t}\lambda(f_sx)\,ds \geq \eta \text{ for } r\in (0,t]\right\}$$
	to be the class of ``good'' orbit segments at scale $\eta$. From here, the decomposition of a given segment $(x,t)$ is defined by taking the prefix to be the largest initial segment in $B(\eta)$, the suffix to be the largest remaining terminal segment in $B(\eta)$, and the good part is the remaining segment, which lies in $\GGG(\eta)$.
\end{defn}

Not every application of \cite{CT} has been a $\lambda$-decomposition, but many are able to be studied using this theory. For more references, see \cite{BCFT,CKP,CFT2,CFT,Wa19}. We make use of them in this paper as they behave well for the necessary pressure estimates, and they also induce natural decompositions in the product space via $\tilde{\lambda}(x,y) = \lambda(x)\lambda(y)$, both of which allow us to make use of Theorem \ref{NewLedr}.

\subsubsection{Pressure}
For any collection of orbit segments $\CCC$, we are able to define its topological pressure, which will be denoted $P(\CCC,\ph)$. In particular, if $\CCC = A\x [0,\infty)$, then $P(\CCC,\ph)$ is precisely the \emph{upper-capacity pressure} as defined in \cite{Pes}, and will be written $P(A,\ph)$. For any $\eps > 0$ and $t > 0$, we say that a set $A$ is \emph{$(t,\eps)$-separated} if given $x,y\in A$, there exists $s\in [0,t]$ such that $d(f_sx,f_sy) \geq \eps$. Then, we define
$$\Lambda_t(\CCC,\ph,\eps) = \sup\left\{\sum_{x\in E_t}e^{\int_{0}^{t}\ph(f_sx)\,ds}\mid E_t\subset\CCC_t \text{ is } (t,\eps)\text{-seperated}\right\}$$
where $\CCC_t = \{x\mid (x,t)\in\CCC\}$ is the set of orbit segments in $\CCC$ of length exactly $t$. With this, we define the \emph{pressure of $\CCC$ at scale $\eps$} by
$$P(\CCC,\ph,\eps) = \limsup_{t\to\infty}\frac{1}{t}\log\Lambda_t(\CCC,\ph,\eps)$$
and the \emph{pressure of $\CCC$} as
$$P(\CCC,\ph) = \lim\limits_{\eps\to 0}P(\CCC,\ph,\eps).$$
When $\ph = 0$ and $A\subset X$, we will write $P(A,0,\eps) = h(A,\eps)$, as this is the topological entropy of $A$. In \cite[Theorem 3.6]{CaT}, the problem of computing the pressure of $B(\eta)$ was simplified in the setting when the entropy map is upper semicontinuous (see Proposition \ref{prop: pressure decrease}). In \S \ref{Pressure} we simplify this further to studying the pressure of a particular compact, invariant set.

\subsubsection{Specification and the Bowen property}
We now briefly discuss the nonuniform versions of the Bowen property and specification that are used in this paper. Let $\CCC\subset X\x [0,\infty)$ be a collection of orbit segments.
\begin{defn}
We say that $\ph$ has the Bowen property on $\CCC$ at scale $\eps > 0$ if
$$\sup_{(x,t)\in\CCC}\sup \left\{\abs{\int_{0}^{t}\ph(f_sx) - \ph(f_sy)\,ds} : d(f_sx,f_sy) \leq \eps \text{ for } s\in [0,t]\right\} < \infty.$$
\end{defn}

\begin{defn}
$\CCC$ has specification at scale $\delta > 0$ if there exists $\tau > 0$ such that for any finite collection of orbit segments $\{(x_i,t_i)\}_{i=1}^n \subset \CCC$, there exists $y\in X$ that $\delta$-shadows each successive orbit segment with gaps of length $\tau$ between each segment. In other words, for all $i$,
$$d\left(f_{s}x_i, f_{s+\sum_{j=1}^{i-1}(\tau + t_i)}y\right) \leq \delta \text{ for } s\in [0,t_i].$$
We say that $\CCC$ has specification if it has specification at scale $\delta$ for all $\delta > 0$.
\end{defn}

These definitions hold in the discrete-time setting as well, with the obvious changes from $t\in [0,\infty)$ to $n\in\N$.

\subsubsection{Expansivity}
We recall a weakening of expansivity introduced in \cite{CT}. While this notion holds in both flows and discrete-time settings, there are notable complications that arise in the flow case.

\begin{defn}
For all $\eps > 0$, define $\Gamma_\eps(x) := \{y\mid d(f_tx,f_ty)\leq \eps\text{ for all }t\in\R\}$ to be the set of points that $\eps$-shadow $x$ for all time. Then define the set of \emph{non-expansive points at scale $\eps$} as
$$\NE(\eps;\FFF) := \{x\in X\mid \Gamma_\eps(x)\not\subset f_{[-s,s]}(x) \text{ for all } s > 0\}$$
where $f_{[-s,s]}(x) = \{y\mid y = f_tx \text{ for some } t\in [-s,s]\}$.
\end{defn}

In contrast, for a homeomorphism $f : X\to X$, we define
$$\NE(\eps;f) = \{x\mid \Gamma_\eps(x)\neq\{x\}\}.$$

In both cases, we then define the \emph{pressure of obstructions to expansivity} as
$$P_{\exp}^\perp(\ph) = \lim\limits_{\eps\downarrow 0}\sup_{\mu}\{P_\mu(\ph)\mid \mu(\NE(\eps)) = 1\}.$$
Observe that this definition is the same in both the discrete-time and flow settings, with only the definition of the non-expansive set changing.

\subsubsection{Conditions for Uniqueness}
We can now state the following result of Climenhaga and Thompson which guarantees the existence of a unique equilibrium state. Rather then present the result in full generality, we phrase it using the terminology and results related to $\lambda$-decompositions, and with the additional assumption of upper semicontinuity of the entropy map. This serves as an illustration of how they can be used to prove uniqueness of equilibrium states, even when proving the $K$-property is beyond the scope of the techniques in this paper.
\begin{thm}[\cite{CT}]\label{thm: CT orig}
Let $\FFF$ be a continuous flow on a compact metric space $X$, and let $\ph : X\to\R$ be continuous. Then, if $P_{\exp}^\perp(\ph) < P(\ph)$ and  $\lambda$ gives rise to a $\lambda$-decomposition such that
\begin{enumerate}
	\item $\GGG(\eta)$ has (possibly weak) specification for all $\eta > 0$;
	\item $\ph$ has the Bowen property on $\GGG(\eta)$ for all $\eta > 0$;
	\item $\sup_\mu\{P_\mu(\ph) \mid \int\lambda\,d\mu = 0\} < P(\ph)$;
\end{enumerate}
and if the entropy map is upper semicontinuous, $(X,\FFF,\ph)$ has a unique equilibrium state.
\end{thm}

We now recall the abstract statement that was shown in \cite{CaT} for the $K$-property.

\begin{thm}[\cite{CaT}]\label{general}
Let $(X,\FFF,\ph)$ be as in Theorem \ref{thm: CT orig}, and suppose that $\GGG(\eta)$ has strong specification for all $\eta > 0$. Furthermore, suppose that $\FFF$ is entropy expansive, and that every equilibrium measure for $\Phi$ is product expansive for $(X \x X, \FFF \x \FFF)$. Then, writing $\tilde{\lambda}(x,y) = \lambda(x)\lambda(y)$, if
$$\sup_\mu\left\{P_\mu(\Phi)\mid\int\tilde{\lambda}\,d\mu = 0\right\} < P(\Phi)$$
then the unique equilibrium state for $(X,\FFF,\ph)$ has the $K$-property.
\end{thm}

This is the abstract result that we will be improving and applying in this paper.

\section{Expansivity}\label{Expansivity}

In Theorem \ref{general}, there are three conditions placed on the expansivity of the flow. One, that the pressure of obstructions to expansivity is smaller than the pressure of the system, is necessary to apply the Climenhaga-Thompson machinery. We will show that the other two conditions, product expansivity of equilibrium states of $(X\x X,\FFF\x\FFF)$ and entropy expansivity of $(X,\FFF)$, can be removed.

\subsection{Product Expansivity}

It can be easily checked that for every flow $(X,\FFF)$ and potential $\ph$, the Cartesian product $(X\x X,\FFF\x\FFF)$ equipped with the potential $\Phi$ never satisfies the inequality $P_{\exp}^\perp(\Phi;\FFF\x\FFF) < P(\Phi)$. This is because $\NE(\eps; \FFF\x \FFF) = X\x X$, which is similarly the reason that the Cartesian product of expansive flows is not expansive.

The notion of \emph{product expansivity} was introduced in \cite{CaT} to rectify this problem. In particular, the \emph{product non-expansive set} is defined as
$$\NE^\x(\eps) := \{(x,y)\mid \Gamma_\eps((x,y)) \not\subset f_{[-s,s]}(x)\x f_{[-s,s]}(y) \text{ for all } s > 0\}.$$
and say that a measure $\mu$ is \emph{product expansive} if $\mu(\NE^\x(\eps)) = 0$ for all small $\eps > 0$.

In Theorem \ref{general}, one needs to show that every equilibrium state for $(X\x X,\FFF\x \FFF,\Phi)$ is product expansive. We will show that the inequality $P_{\exp}^{\perp}(\ph) < P(\ph)$ implies this condition.

\begin{prop}\label{prop: exp gap}
For any continuous flow $(X,\FFF)$, any ergodic measure $\nu$ which is not product expansive satisfies $P_\nu(\Phi) \leq P_{\exp}^\perp(\ph) + P(\ph)$.
\end{prop}

\begin{proof}
Let $\nu\in\MMM(X\x X, \FFF\x\FFF)$, and suppose that $\nu$ is not product expansive. Then for all $\eps > 0$, $\nu(\NE^\x(\eps)) = 1$, since this is an invariant set. Now observe that $\NE^\x(\eps) = X\x \NE(\eps)\cup \NE(\eps)\x X$. By ergodicity of $\nu$ and invariance of $\NE(\eps)$, we can assume without loss of generality that $\nu(X\x \NE(\eps)) = 1$. Writing $(\pi_1)_*\nu(A) = \nu(X\x A)$, we see that $(\pi_1)_*\nu(\NE(\eps)) = 1$. As $\eps$ can be arbitrarily small, it follows that $P_{(\pi_1)_*\nu}(\ph) \leq P_{\exp}^\perp(\ph)$. Then, since $\nu$ is a joining of $(\pi_1)_*\nu$ and $(\pi_2)_*\nu$, Proposition \ref{prop: projection of ES} shows that
$$P_\nu(\Phi) \leq P_{(\pi_1)_*\nu}(\ph) + P_{(\pi_2)_*\nu}(\ph) \leq P_{\exp}^\perp(\ph) + P(\ph).$$
With this, we have completed our proof.
\end{proof}

\begin{cor}\label{cor: prod exp}
If $P_{\exp}^{\perp}(\ph) < P(\ph)$, then any equilibrium state for $(X\x X,\FFF\x\FFF,\Phi)$ is product expansive.
\end{cor}

Consequently, we can remove the product expansivity assumption from Theorem \ref{general}.

\subsection{Ledrappier's Criterion}

In Theorem \ref{Ledr}, Ledrappier gives an elegant criterion for the $K$-property. We will show that the proof he provides actually shows a stronger result than the one he stated. In order to do so, we provide the computations omitted in \cite{L}, and in the process remove the assumption of asymptotic entropy expansiveness.

\begin{thm}\label{NewLedr}
	Let $(X,f)$ be a dynamical system, and let $\ph : X\to\R$ be a continuous function. Let $(X\x X, f\x f)$ be the Cartesian product of $(X,f)$ with itself, and define the potential $\Phi(x,y) = \ph(x) + \ph(y)$. If $\Phi$ has a unique equilibrium state, then $(X,f,\ph)$ has a unique equilibrium state, which has the K-property.
\end{thm}

\begin{proof}
	First, observe by Proposition \ref{prop: projection of ES} that $\ph$ has an equilibrium state if $\Phi$ does. We prove the contrapositive. Suppose $\mu$ is the unique equilibrium state for $(X,f,\ph)$ which is not $K$. Then $\mu\x\mu$ is the unique equilibrium state for $(X\x X,f\x f,\Phi)$. Now, as the Pinsker factor $\Pi$ is non-trivial, we can define the measure $m$ to be the relatively independent self-joining of $\mu$ over the Pinsker factor. For the reader's convenience, we note that this is equivalent to defining $$m(A\x A') = \int_A \E[\chi_{A'}\mid\Pi]\,d\mu.$$	
	Observe that
	$$\int\Phi\,dm = 2\int\ph\,d\mu,$$
	because $\Phi$ acts independently on each coordinate, and $m(A\x X) = m(X\x A) = \mu(A)$. We now will show that
	$$h_m(f\x f) = 2h_\mu(f).$$
	To compute the entropy of $m$, we appeal to Proposition \ref{thm: fiber entropy} and the definition of fiber entropy. Let $\xi$ be a partition of $X$. Conditioning on a factor does not increase entropy, so
	$$h_m(\xi\x \xi; f\x f) \geq h_m(\xi\x\xi \mid \Pi) = \int_{y\in X}H_{m}(\xi\x\xi \mid y)\,d\mu(y)$$
	where the integrand in the last term is fiber entropy. Then, as the disintegration of $m$ over $\Pi$ is given by $\mu_y\times \mu_y$, where $\{\mu_y\}$ is the disintegration of $\mu$ over $\Pi$, observe
	\begin{align*}
	H_{m}(\xi\x\xi\mid y) &= \lim\limits_{n\to\infty}H_{\mu_y\x\mu_y}\left(\xi\x\xi\mid \bigvee_{i=0}^{n-1}(f\x f)^{-i}(\xi\x\xi)\right)
	\\
	&= \lim\limits_{n\to\infty}2H_{\mu_y}\left(\xi\mid \bigvee_{i=0}^{n-1}f^{-i}\xi\right)
	\\
	&= 2H_\mu(\xi \mid y).
	\end{align*}
	Consequently, $h_m(\xi\mid \Pi) = 2h_\mu(\xi\mid \Pi)$. Taking the supremum over all finite partitions, we have that
	$$h_m(f\x f) = \sup_{\xi} h_m(\xi\x\xi) \geq 2\sup_{\xi} h_\mu(\xi \mid \Pi) = 2h_\mu(f\mid \Pi) = 2h_\mu(f)$$
	where the last equality follows because $\Pi$ is a zero entropy factor \cite[Fact 4.1.6]{downarowicz}.
	Therefore, we see that
	$$P(\Phi) \geq P_m(\Phi) \geq 2P_\mu(\ph) = 2P(\ph) = P(\Phi).$$
	Thus, $m$ is an equilibrium state for $\Phi$. As $\mu\x\mu$ is an equilibrium state for $\Phi$ by assumption, this proves the contrapositive.
\end{proof}

This, along with the fact that $P(\Phi,\gamma) = P(\Phi)$ for small $\gamma$ by an easy modification of \cite[Proposition 3.7]{CT}, allows us to remove the assumption of entropy expansiveness from Theorem \ref{general}.

\section{Pressure Estimates}\label{Pressure}

We recall the two pressure conditions in Theorem \ref{general}. In particular, for a $\lambda$-decomposition, writing $\tilde{\lambda}(x,y) = \lambda(x)\lambda(y)$, we wish to show that
$$\sup_\mu\left\{P_\mu(\ph)\mid\int\lambda\,d\mu = 0\right\} < P(\ph)$$
and
$$\sup_\mu\left\{P_\mu(\Phi)\mid \int\tilde{\lambda}\,d\mu = 0\right\} < P(\Phi).$$
We will show that the pressure gap in the product space implies the corresponding inequality in the base system. While we do not know if the reverse implication holds---which would be extremely interesting---we do provide two separate methods of checking that the pressure gap in the product holds. Our first approach requires some knowledge of the unique equilibrium state in the base, and will be used for perturbations of Ma\~n\'e diffeomorphisms, and the second is a folklore result, which we will apply to the Katok map.

Recall the method for computing pressure estimates for $\lambda$-decompositions.

\begin{prop}[{\cite[Theorem 3.6]{CaT}}]\label{prop: pressure decrease}
For any $\lambda$-decomposition, if the entropy map is upper semicontinuous
$$\lim\limits_{\eta\downarrow 0}P(B(\eta),\ph) \leq \sup\left\{P_\nu(\ph)\mid \int\lambda\,d\nu = 0\right\}.$$
\end{prop}

Writing $B_\infty := \bigcap_{t\in\R}f_t\lambda^{-1}(0)$ and $\tilde{B}_\infty$ for the corresponding set for $\tilde{\lambda}$, we can restate this result into a ``pressure gap'' formulation.

\begin{prop}\label{prop: pressure gap}
For any lower semicontinuous $\lambda : X\to [0,\infty)$,
$$\sup\left\{P_\nu(\ph)\mid \int\lambda\,d\nu = 0\right\} = P(B_\infty,\ph).$$
\end{prop}

\begin{proof}
First observe that $B_\infty = \bigcap_{t\in\Q}f_t\lambda^{-1}(0)$ by continuity of the flow. Additionally, by lower semicontinuity and non-negativity, $\lambda^{-1}(0)$ is compact. Therefore, $B_\infty$ is a countable intersection of compact sets, and so, compact. Furthermore, by definition, $B_\infty$ is flow-invariant. Consequently, by the Variational Principle, we see
$$P(B_\infty,\ph) = \sup\{P_\nu(\ph)\mid \nu\in\MMM(B_\infty,\FFF)\}.$$
Now observe that $\nu\in\MMM(B_\infty,\FFF)$ if and only if $\int\lambda\,d\nu = 0$.
\end{proof}

\begin{cor}\label{cor: pressure gap sufficient}
For any $\lambda$-decomposition, if $P(B_\infty,\ph) < P(\ph)$, then there exists $\eta$ so that $P(B(\eta),\ph) < P(\ph)$, so long as the entropy map is upper semicontinuous.
\end{cor}

\begin{rem}
In the case when $B_\infty$ is empty, it is easily seen that for all sufficiently small $\eta$, $B(\eta)$ does not contain arbitrarily long orbit segments, and so the inequality is satisfied trivially.
\end{rem}

We can now show that a pressure gap in the product implies a pressure gap in the base.

\begin{prop}\label{prop: product gap sufficient}
Let $\lambda$ give rise to a $\lambda$-decomposition. Then 
$$P(\tilde{B}_\infty,\Phi) < P(\Phi) \implies P(B_\infty,\ph) < P(\ph)$$
where $\tilde{B}_\infty = \bigcap_{t\in\R}f_t\tilde{\lambda}^{-1}(0)$.
\end{prop}

\begin{proof}
Let $\nu\in\MMM(B_\infty,\FFF)$. Then $\int\lambda\,d\nu = 0$, and so $\int\tilde{\lambda}\,d(\nu\x\nu) = 0$. Hence, $\nu\x\nu\in\MMM(\tilde{B}_\infty,\FFF\x\FFF)$. Consequently, if $P(\tilde{B}_\infty,\Phi) < P(\Phi)$, then we see
$$2P(\ph) > P(\tilde{B}_\infty,\Phi) \geq P_{\nu\x\nu}(\Phi) = 2P_\nu(\ph).$$
Applying the Variational Principle to $(B_\infty,\FFF)$ completes the proof.
\end{proof}

Combining this with Theorem \ref{general} and Corollary \ref{cor: prod exp}, we obtain the following result which directly implies Theorem \ref{Theorem A}.

\begin{thm}\label{thm: main thm}
Let $(X,\FFF)$ be a continuous flow on a compact metric space and $\ph : X\to\R$ a continuous potential. Suppose $P_{\exp}^\perp(\ph) < P(\ph)$ and $\lambda$ gives rise to a $\lambda$-decomposition such that 
\begin{enumerate}
	\item $\GGG(\eta)$ has specification for all $\eta > 0$;
	\item $\ph$ has the Bowen property on $\GGG(\eta)$ for all $\eta > 0$;
\end{enumerate}
while the product system $(X\x X,\FFF\x\FFF,\Phi)$ satisfies
\begin{enumerate}
	\item $P(\tilde{B}_\infty,\Phi) < 2P(\ph)$;
	\item The entropy map is upper semicontinuous.
\end{enumerate}
Then the unique equilibrium state for $(X,\FFF,\ph)$ has the $K$-property.
\end{thm}

Of particular note is that any pressure gap is a $C^0$-open condition on the space of potentials. This fact is encapsulated in the following proposition.

\begin{prop}\label{prop: open pressure}
	Suppose $A$ is compact and invariant, and $P(A,\ph) < P(\ph)$. Then for all continuous $\psi$ such that $2\norm{\ph - \psi} < P(\ph) - P(A,\ph)$ and $c\in\R$ we have $P(A,\psi + c) < P(\psi + c)$.
\end{prop}

\begin{proof}
	Recall that $P(A,\psi + c) = P(A,\psi) + c$, so it suffices to prove this for $c = 0$. For all potentials $\psi$, $\abs{P(\ph) - P(\psi)} \leq \norm{\ph - \psi}$ by \cite[Theorem 9.7.iv]{Wa}. This holds for the pressure of any compact, invariant set as well, and so for any $\psi$ sufficiently close to $\ph$, we see that
	$$P(A,\psi) \leq P(A,\ph) + \norm{\ph - \psi} < P(\ph) - 2\norm{\ph - \psi} + \norm{\ph - \psi} \leq P(\psi).   \qedhere$$
\end{proof}
Consequently, showing a pressure gap for the measure of maximal entropy implies the pressure gap for all potentials sufficiently close to constant.

\subsection{Measure Estimate}

\begin{thm}\label{thm: measure estimate}
	Let $\mu$ be the unique equilibrium state of $(X,\FFF,\ph)$, and suppose that $\mu(\lambda^{-1}(0)) < \frac{1}{2}$. If the entropy map is upper semicontinuous on the product system, then $P(\tilde{B}_\infty,\Phi) < 2P(\ph)$.
\end{thm}

\begin{proof}
	Let $\nu$ be an equilibrium state for $(X\x X,\FFF\x\FFF)$ equipped with potential $\Phi$. By Proposition \ref{prop: projection of ES}, the projections onto each coordinate, $\nu_i := \nu\circ\pi_i^{-1}$ are both equilibrium states for $(X,\FFF,\ph)$. Therefore, $\nu_i = \mu$ for $i=1,2$. Consequently, as
	$$\tilde{B}_\infty = \bigcap_{t\in\R}(f_t\x f_t)(\lambda^{-1}(0)\x X \cup X\x\lambda^{-1}(0)) \subset \lambda^{-1}(0)\x X\cup X\x\lambda^{-1}(0)$$
	we have
	$$\nu(\tilde{B}_\infty) \leq \nu(\lambda^{-1}(0)\x X) + \nu(X\x \lambda^{-1}(0)) = \nu_1(\lambda^{-1}(0)) + \nu_2(\lambda^{-1}(0)) = 2\mu(\lambda^{-1}(0)).$$
	Therefore, if $\mu(\lambda^{-1}(0)) < \frac{1}{2}$, then $\nu(\tilde{B}_\infty) < 1$. Finally, let $m$ be an equilibrium state for $(\tilde{B}_\infty,\FFF\x\FFF,\Phi)$. This exists as $\tilde{B}_\infty$ is compact and invariant, and the entropy map is upper semicontinuous. Then $m(\tilde{B}_\infty) = 1$, and so $P(\tilde{B}_\infty,\Phi) = P_m(\Phi) < P(\Phi)$, completing our proof.
\end{proof}

This approach is particularly useful in the setting of Ma\~n\'e diffeomorphisms, which we will discuss further in \S \ref{Applications}.

\subsection{Entropy Production Argument}

The main result of this section is an argument that creates an entropy gap in systems with global specification. This is a folklore result, and we limit ourselves to the MME case. However, it should be thought of as a blueprint for general pressure production arguments, with weaker assumptions on the system, such as non-uniform specification and non-zero potentials. This style of argument has been previously carried out in \cite{CFT} and \cite{BCFT}. However, those arguments were tailored to the specific settings being considered at the time, and we believe there is value to the general statement presented here.

\begin{thm}\label{prop: entropy production}
Suppose $(X,\FFF)$ has specification and let $A$ be a proper subset of $X$ which is compact and invariant. Take $\eps$ such that there exists $y\notin A$ with $d(A,y) > 3\eps$. If there exists $C > 0$ so that for every maximal $(t,3\eps)$-separated set $E_t$, we have $\# E_t \geq Ce^{th(A)}$, then $h(A) < h(X)$.
\end{thm}

\begin{rem}
The above condition on $\# E_t$ is satisfied if, for instance, $h(A,6\eps) = h(A)$ \cite[Lemmas 4.1 and 4.2]{CT}.
\end{rem}

\begin{proof}
Let $\tau$ be the specification constant for $\eps$, let $T > 2\tau+1$, and take $\alpha > 0$ to be small and rational. Choose $N\in\N$ to be large, so that $\alpha N\in\N$. We will now show that by ``interweaving'' between $A$ and $X\setminus A$ for a set of prescribed times, we can create a $(NT,\eps)$-separated set whose partition sum is larger than the entropy of $A$ for arbitrarily large $N$.

Choose $\mathcal{I}\subset \{T,2T,\cdots, (N-1)T\}$ such that $\#\mathcal{I} = \alpha N - 1$. This will be one set of interweaving times, and we will define a set of $(NT,\eps)$-separated points based on it. Writing $\mathcal{I} = \{j_1T,j_2T,\cdots, j_{\alpha N - 1}T\}$, set $k_1 = j_1T$, $k_{\alpha_N} = NT - (2\tau + 1)$ and $k_i = (j_i - j_{i-1})T - (2\tau + 1)$ for $2\leq i \leq \alpha N - 1$. For each $k_i$, define $E_{k_i}$ to be a $(k_i,3\eps)$-separated subset of $A$ of maximal cardinality. Then, define the map $\Pi : E_{k_1}\x\cdots \x E_{k_{\alpha N}}\to X$ to be a point $z$, guaranteed by specification, that $\eps$-shadows $(x_1,k_1)$, followed by $(y,1)$, followed by $(x_2,k_2)$, and so on. As $\sum_{i=1}^{\alpha N}k_i + (\alpha N - 1)(2\tau + 1) = NT$, and each $E_{k_i}$ is $(k_i,3\eps)$-separated, it follows that the image of $\Pi$ is $(NT,\eps)$-separated. For ease of notation, we will refer to this $(NT,\eps)$-separated set by $E_{\mathcal{I}}$.

We now need to show that given two different sets of interweaving times, $\mathcal{I}\neq \mathcal{I}'$, the union $E_{\mathcal{I}}\cup E_{\mathcal{I}'}$ is $(NT,\eps)$-separated. Consider $j_iT\in \mathcal{I}\setminus \mathcal{I}'$. Then, given $w\in E_{\mathcal{I}}$, by construction, $d(f_{j_iT + \tau}w,y) \leq \eps$. However, for all $z\in E_{\mathcal{I}'}$, $d(f_{j_iT + \tau}z, A) \leq \eps$ as $A$ is invariant. Therefore, as $d(y,A) \geq 3\eps$, it follows that for all such $z$, $d_{NT}(w,z) \geq \eps$. Thus, we have shown that $E_{\mathcal{I}}\cup E_{\mathcal{I}'}$ is $(NT,\eps)$-separated for any two distinct sets of interweaving times.

By assumption, there exists $C > 0$ such that $\# E_k \geq Ce^{kh(A)}$ for all $k$. Therefore,
$$\# E_{\mathcal{I}} = \prod_{i=1}^{\alpha N} \# E_{k_i} \geq C^{\alpha N}e^{(NT - (\alpha N - 1)(2\tau + 1))h(A)} = C^{\alpha N}e^{NTh(A)} e^{-(\alpha N - 1)(2\tau + 1)h(A)}.$$
As there are $\binom{N}{\alpha N} \geq \alpha e^{-N\alpha\log \alpha}$ different ways to choose a set $\mathcal{I}$, we see that
$$\# \bigcup_{\mathcal{I}} E_{\mathcal{I}} \geq C^{\alpha N}\alpha e^{NTh(A)} e^{-(\alpha N - 1)(2\tau + 1)h(A)}e^{-N\alpha\log\alpha}.$$
Because $\bigcup_{\mathcal{I}}E_{\mathcal{I}}$ is $(NT,\eps)$-separated, we see that
$$h(X,\eps) \geq \lim\limits_{N\to\infty}\frac{1}{NT}\log\#\bigcup_{\mathcal{I}} E_{\mathcal{I}} \geq h(A) - \frac{\alpha(2\tau + 1)h(A)}{T} - \frac{\alpha\log\alpha}{T} + \frac{\alpha\log C}{T}.$$
Since $\alpha$ is chosen independently, by taking $\alpha < Ce^{-(2\tau + 1)h(A)}$, we see that $h(X) \geq h(X,\eps) > h(A)$, as desired.
\end{proof}
This result holds in the discrete-time setting as well (with a simplified proof), and we present the following corollary in that setting, as we will use it in \S \ref{Applications}.
\begin{cor}\label{cor: pressure gap}
If $(X,f)$ is expansive and has specification, and $\lambda$ gives rise to a $\lambda$-decomposition with $B_\infty \subsetneq X$, then $h(B_\infty) < h(f)$.
\end{cor}

\section{Discrete-time Cases}\label{Discrete}
The original version of this theorem, as well as the improvements in this paper, have all been in the continuous time setting. However, all of the arguments carry over easily to the discrete time setting, and in some cases, simplify considerably. We will show that the expansivity issues that arise in the flow case are no longer problematic. Consequently, we are able to directly apply the discrete-time theorem from \cite{CT} in the product system, rather than having to adapt arguments and technical expansivity lemmas as we do in the flow setting. Recall the theorem that we wish to apply.

\begin{thm}[\cite{CT}]\label{thm: CT discrete}
Let $X$ be a compact metric space, $f : X\to X$ a homeomorphism, and $\ph : X\to\R$ a continuous potential. Suppose that $P_{\exp}^\perp(\ph) < P(\ph)$ and that there exists a decomposition $(\PPP,\GGG,\SSS)$ such that
\begin{enumerate}
	\item $\GGG$ has specification at all scales
	\item $\ph$ has the Bowen property on $\GGG$
	\item $P(\PPP\cup\SSS,\ph) < P(\ph)$.
\end{enumerate}
Then there is a unique equilibrium state $\mu$ for $(X,f,\ph)$. 
\end{thm}

We will first address expansivity concerns, then show that the decomposition lifts to the product, and finally, examine what considerations can be made to attain the required pressure estimates in the product system.

\subsection{Expansivity}

One of the fundamental difficulties in the flow setting boils down to the fact that the Cartesian product of an expansive flow with itself is not expansive. This is no longer an issue in the discrete-time setting. To see why, consider the definition of a non-expansive point.

\begin{defn}
Given $f : X\to X$ and $\eps > 0$, the non-expansive set at scale $\eps$ is
$$\NE(\eps;f) := \{x\mid \Gamma_\eps(x)\neq \{x\}\}$$
where $\Gamma_\eps(x) := \{y\mid d(f^nx,f^ny) \leq \eps \text{ for all } n\in \ZZ\}$ is the set of all points which $\eps$-shadow $x$ for all time.
\end{defn}

Unlike in the flow case, where $\NE(\eps;\FFF\x\FFF) = X\x X$, we are able to express the non-expansive set in the product space in terms of the non-expansive set in the base.
\begin{prop}
Given $f : X\to X$ and $\eps > 0$,
$$\NE(\eps;f\x f) = \left(X\x \NE(\eps;f)\right) \cup \left(\NE(\eps;f)\x X\right).$$
\end{prop}
\begin{proof}
The key to this proof is the observation that $\Gamma_\eps((x,y)) = \Gamma_\eps(x)\x \Gamma_\eps(y)$. Using this, we see that given $(x,y)\in \NE(\eps;f\x f)$, then
$$(x,y) \neq \Gamma_\eps((x,y)) = \Gamma_\eps(x)\x \Gamma_\eps(y).$$
Therefore, either $x$ or $y$ is in $\NE(\eps;f)$. Similarly, if $x\in \NE(\eps;f)$, then for all $y\in X$, $\Gamma_\eps((x,y)) \neq \{x,y\}$. This completes our proof.
\end{proof}

Using this, we can show that any equilibrium measure on $(X\x X,f\x f,\Phi)$ ``does not see'' the non-expansive set.

\begin{prop}\label{prop: discrete exp gap}
If $P_{\exp}^{\perp}(\ph) < P(\ph)$, then $P_{\exp}^{\perp}(\Phi) < P(\Phi)$.
\end{prop}

\begin{proof}
The proof follows that of Proposition \ref{prop: exp gap}, merely replacing the flow $\FFF$ with the map $f : X\to X$. The key observation is that if $\nu\in\MMM(X\x X, f\x f)$ is a measure such that $\nu(\NE(\eps;f\x f)) = 1$ for some $\eps > 0$, then for $i = 1$ or $2$, we have $(\pi_i)_*\nu(\NE(\eps;f)) = 1$. From here, the pressure inequality in the base implies the pressure inequality in the product.
\end{proof}

\subsection{$\lambda$-Decomposition}
We adapt the definition of $\lambda$-decompositions to the discrete time setting, carrying everything over as written. Let $\lambda : X\to [0,\infty)$ be lower semicontinuous and bounded. Then for all $\eta > 0$, define the set of bad orbit segments by
$$B(\eta) = \left\{(x,n)\mid \frac{1}{n}\sum_{i=0}^{n-1}\lambda(f^ix) < \eta\right\}$$
and the set of good orbit segments by
$$\GGG(\eta) = \left\{(x,n)\mid \frac{1}{k}\sum_{i=0}^{k-1}\lambda(f^ix) \geq \eta \text{ and } \frac{1}{k}\sum_{i=n-k}^{n-1}\lambda(f^ix)\geq \eta \text{ for } 1\leq k \leq n\right\}.$$
Then, the decomposition of an orbit segment $(x,n)$ is created by first taking the largest initial segment in $B(\eta)$, calling that the prefix, then taking the largest terminal segment in $B(\eta)$, and calling that the suffix. The remaining segment lies in $\GGG(\eta)$.

Just as in the flow case, $\lambda$-decompositions lift nicely to the product, by defining $\tilde{\lambda} : X\x X\to [0,\infty)$ by $\tilde{\lambda}(x,y) = \lambda(x)\lambda(y)$, and then studying the corresponding $\tilde{\lambda}$-decomposition.

\begin{prop}\label{prop: lift to products}
Suppose that $\GGG(\eta)$ is the set of good orbit segments for the $\lambda$-decomposition of $(X,f)$. Then the set of good orbit segments in the product space with respect to the function $\tilde{\lambda}$, written $\tGGG(\eta)$, satisfies
$$\tGGG(\eta)\subset \GGG\left(\frac{\eta}{\norm{\lambda}}\right)\x \GGG\left(\frac{\eta}{\norm{\lambda}}\right)$$
for all $\eta \geq 0$.
\end{prop}

\begin{proof}
Let $((x,y),n)\in\tGGG(\eta)$. Then for all $1\leq k\leq n$, we see
$$\frac{1}{k}\sum_{i=1}^{k}\lambda(f^ix) \geq \frac{1}{k}\sum_{i=1}^{k}\lambda(f^ix)\frac{\lambda(f^iy)}{\norm{\lambda}} = \frac{1}{\norm{\lambda}k}\sum_{k=1}^{k}\tilde{\lambda}(f^ix,f^iy) \geq \frac{\eta}{\norm{\lambda}}.$$
A similar computation holds for the average along terminal subsegments, as well as in the other coordinate, and so our proof is complete.
\end{proof}

\begin{rem}
We should not expect to have equality here, as given $((x,y),n)\in \GGG(\eta)\x\GGG(\eta)$, there is no reason to expect both coordinates to experience ``good'' behavior with respect to $\lambda$ at the same time, and $\tilde{\lambda}$ is constructed to always identify ``bad'' behavior in one coordinate as bad for the entire segment.
\end{rem}

It is classical that specification and the Bowen property lift to products. Hence, from Proposition \ref{prop: lift to products}, we have the following corollaries.

\begin{cor}
If $\GGG(\eta)$ has specification at scale $\delta$, then $\tGGG(\eta)$ has specification at scale $\delta$ as well.
\end{cor}

\begin{cor}
If $\ph$ has the Bowen property on $\GGG(\eta)$, then $\Phi$ has the Bowen property on $\tGGG(\eta)$.
\end{cor}

Beyond lifting to products, much of the strength of $\lambda$-decompositions lies in the fact that there is an easy way to compute the pressure of the collection of bad orbit segments, by sending $\eta$ to $0$.

\begin{prop}\label{prop: discrete pressure decrease}
If the entropy map is upper semicontinuous, then $$\lim\limits_{\eta\downarrow 0}P(B(\eta),\ph) \leq \sup\left\{P_\mu(\ph)\mid \int\lambda\,d\mu = 0\right\} = P(B_\infty,\ph),$$
where $B_\infty = \bigcap_{n\in\ZZ}f^n\lambda^{-1}(0).$
\end{prop}

\begin{proof}
The first inequality follows the exact proof of \cite[Theorem 3.6]{CaT}, and the second equality is shown in Proposition \ref{prop: pressure gap}.
\end{proof}

Combining all these results with the discrete-time version of Proposition \ref{prop: product gap sufficient}, we obtain the discrete-time analogue of Theorem \ref{thm: main thm}, which implies Theorem \ref{Theorem B}.

\begin{thm}
Let $X$ be a compact metric space, $f : X\to X$ a homeomorphism, and $\ph : X\to\R$ a continuous potential. Then if $P_{\exp}^{\perp}(\ph) < P(\ph)$ and $\lambda$ yields a $\lambda$-decomposition satisfying the following properties:
\begin{enumerate}
	\item $\GGG(\eta)$ has specification at all scales for all sufficiently small $\eta > 0$;
	\item $\ph$ has the Bowen property on $\GGG(\eta)$ for all sufficiently small $\eta > 0$;
\end{enumerate}
and the product system has the following properties:
\begin{enumerate}
	\item $P(\bigcap_{n\in\ZZ} f^n\lambda^{-1}(0)\x X \cup X\x f^n\lambda^{-1}(0)) < 2P(\ph)$;
	\item The entropy map is upper semicontinuous
\end{enumerate}
then the unique equilibrium state for $(X,f,\ph)$ has the $K$-property.
\end{thm}

\begin{proof}
Because the entropy map on the product is upper semicontinuous, by Proposition \ref{prop: pressure decrease}, for all sufficiently small $\eta > 0$, $P(B(\eta),\Phi) < P(\Phi)$. Then, $\tGGG(\eta)$ has specification at all scales for all small $\eta > 0$, and $\Phi$ has the Bowen property on $\tGGG(\eta)$ as well. Finally, $P_{\exp}^{\perp}(\Phi) < P(\Phi)$ by Proposition \ref{prop: discrete exp gap}, and so we can apply Theorem \ref{thm: CT discrete} to establish that $(X\x X,f\x f,\Phi)$ has a unique equilibrium state. Therefore, by Proposition \ref{NewLedr}, we have established the $K$-property in the base.
\end{proof}

As in the flow case, we can provide some conditions for pressure on the base system that imply the necessary pressure gap in the product. In particular, Proposition \ref{prop: open pressure} holds to show that the pressure gap is an open condition. Additionally, Theorems \ref{thm: measure estimate} and \ref{prop: entropy production} hold in the discrete-time setting, with the proof simplifying in the case of Theorem \ref{prop: entropy production}. For brevity, we do not restate them here, but simply refer the reader to $\S$ \ref{Pressure}.

\subsection{One-Sided $\lambda$-Decompositions}
In some situations, the most useful decomposition for a system is not a $\lambda$-decomposition, but rather what we define in this paper as a \emph{one-sided $\lambda$-decomposition}. These differ from $\lambda$-decompositions in that either the set of prefixes or the set of suffixes is empty, with a corresponding change in the definition of what a good orbit segment is. The proofs of the various results for these decompositions differ from those for $\lambda$-decompositions in a technical nature only, and in spirit, are exactly the same. Consequently, we use the same notation for $\GGG(\eta)$, and use context to differentiate the setting.

\begin{defn}
Let $\lambda : X\to [0,\infty)$ be a bounded, lower semicontinuous function. For all $\eta\in [0,1]$, let $B(\eta)$ as in the definition of a $\lambda$-decomposition. Define
$$\GGG(\eta) := \left\{(x,n)\mid \frac{1}{k}\sum_{i=0}^{k-1}\lambda(f^ix) \geq \eta \text{ for all } 1 \leq k \leq n\right\}$$
to be the collection of good orbit segments. We define the one-sided $\lambda$-decomposition (with prefixes) as follows: given $(x,n)$, take the largest $k \leq n$ such that $(x,k)\in B(\eta)$. Then $(x,k)\in\PPP$, and $(f^kx,n - k)\in \GGG(\eta)$.
\end{defn}

\begin{rem}
A one-sided $\lambda$-decomposition with suffixes is obtained in the following manner. Given $(x,n)$, take the largest $k\leq n$ so that $(f^{n-k}x, k)\in B(\eta)$ to be the suffix, and then $(x,n-k)$ will be in a ``reversed'' $\GGG(\eta)$ from above, where every terminal subsegment has average at least $\eta$. In this section, we will work in the setting with prefixes, as all proofs are analogous.
\end{rem}

Observe that one-sided $\lambda$-decompositions enjoy the same nice properties as $\lambda$. In particular, they lift naturally to products.

\begin{prop}
Let $\GGG(\eta)$ be the set of good orbit segments for a one-sided $\lambda$-decomposition and let $\tGGG(\eta)$ be the set of good orbit segments for a one-sided $\tilde{\lambda}$-decomposition with $\tilde{\lambda}(x,y) = \lambda(x)\lambda(y)$. Then
$$\tGGG(\eta) \subset\GGG\left(\frac{\eta}{\norm{\lambda}}\right)\x \GGG\left(\frac{\eta}{\norm{\lambda}}\right).$$
\end{prop}

\begin{proof}
The proof is the exact same as the written computation in Proposition \ref{prop: lift to products}.
\end{proof}

\begin{cor}
If $\GGG(\eta)$ has specification for all $\eta > 0$ and $\ph$ has the Bowen property on $\GGG(\eta)$ for all $\eta > 0$, then $\tGGG(\eta)$ has specification for all $\eta > 0$, and $\Phi(x,y) = \ph(x) + \ph(y)$ has the Bowen property on $\tGGG(\eta)$.
\end{cor}

Note that the pressure estimates discussed earlier in the paper hold as well, because $\PPP = B(\eta)$, the definition of which does not change in the setting of one-sided $\lambda$-decompositions. Furthermore, $P_{\exp}^{\perp}(\ph)$ is independent of the decomposition, and so lifts to the product as well. Therefore, we have the following theorem.

\begin{thm}\label{thm: half lambda}
Let $X$ be a compact metric space, $f : X\to X$ a homeomorphism, and $\ph: X\to \R$ a continuous potential. Suppose $P_{\exp}^\perp(\ph) < P(\ph)$ and $\lambda$ is a bounded, non-negative, lower semicontinuous function such that the corresponding $\lambda$-decomposition satisfies:
\begin{enumerate}
	\item $\GGG(\eta)$ has specification at all scales for all sufficiently small $\eta > 0$;
	\item $\ph$ has the Bowen property on $\GGG(\eta)$ for all sufficiently small $\eta > 0$;
\end{enumerate}
and the product system has the following properties:
\begin{enumerate}
	\item $P(\bigcap_{n\in\ZZ} f^n\lambda^{-1}(0)\x X\cup X\x f^{n}\lambda^{-1}(0)) < 2P(\ph)$;
	\item The entropy map is upper semicontinuous.
\end{enumerate}
Then the unique equilibrium state for $(X,f,\ph)$ has the $K$-property.
\end{thm}

\begin{rem}
One-sided $\lambda$-decompositions can be formulated for flows in the obvious way. Then, the analogue of this theorem holds for flows. In other words, Theorem \ref{Theorem A} holds for one-sided $\lambda$-decompositions. We omit further discussion, as no new technical difficulties arise.
\end{rem}

\section{Applications}\label{Applications}

The weaker version of this result has already been applied in some settings. In \cite{CaT}, it was used to show the $K$-property for some potentials for the geodesic flow on rank 1 compact manifolds with non-positive curvature. Additionally, in \cite{CKP}, it was used to establish the $K$-property in the same setting except for manifolds with no focal points. However, there are other settings in which this result can be applied. For instance, work of Climenhaga, Fisher, and Thompson on Ma\~n\'e diffeomorphisms \cite{CFT2} uses a decomposition similar to one-sided $\lambda$-decompositions, and work of Wang on equilibrium states for the Katok map \cite{Wa19} also uses the setup of $\lambda$-decompositions. In what follows, we will show how this work can be applied in these settings as well to obtain the $K$-property.

{\subsection{Katok Map}

The Katok map, introduced in \cite{katokmap}, is an example of a $C^\infty$, non-uniformly hyperbolic diffeomorphism on $\mathbb{T}^2$, which is created by a slow-down of a uniformly hyperbolic system in a small neighborhood of a fixed point. We make use of the $\lambda$-decomposition formalism for the Katok map introduced in \cite{Wa19}, which take $\lambda$ to be the indicator function of the complement of this neighborhood. Shahidi and Zelerowicz \cite{ShaZel} have already shown that a wide class of equilibrium states of the Katok map are Bernoulli, and thus $K$, so we mention the application only briefly, as a different method for achieving the $K$-property, and as an indication for the strength of this result as it applies to perturbations of potentials. We refer to \cite{Wa19} for the precise definitions and many results about the Katok map.

The Katok map has specification and is expansive, so its Cartesian product $(X\x X,f\x f)$ is as well. Therefore, by Corollary \ref{cor: pressure gap}, 
$$h(\tilde{B}_\infty) < h(X\x X) = 2h(X)$$
where $\tilde{B}_\infty$ is obtained from the $\lambda$-decomposition in \cite{Wa19}. Consequently, by Proposition \ref{prop: open pressure}, this pressure gap holds for all potentials sufficiently close to constant, and so, for any such potential which has the Bowen property on this decomposition, the unique equilibrium state has the $K$-property. Using \cite[Proposition 6.3]{Wa19}, this applies to all H\"older continuous potentials sufficiently close to constant.}

\subsection{Ma\~n\'e Diffeomorphisms}

We show that our work applies to the class of Ma\~n\'e diffeomorphisms, proving Theorem \ref{Theorem C}. These are partially hyperbolic, robustly transitive diffeomorphisms originally constructed in \cite{Ma78} on $\mathbb{T}^3$, although the construction applies for $d\geq 3$ as well. In \cite{BFSV}, the authors showed intrinsic ergodicity in this setting through constructing the measure of maximal entropy. In \cite{CFT2}, the authors obtain uniqueness of equilibrium states for a class of H\"older potentials and suitable $C^1$ perturbations of these diffeomorphisms using the decomposition theory developed in \cite{CT}. We will use these results in combination with the results established in this paper to establish the $K$-property for a class of Ma\~n\'e diffeomorphisms sufficiently $C^0$-close to Anosov and H\"older potentials sufficiently close to constant. In \cite{BFSV}, they obtain the Bernoulli property for the measures of maximal entropy, but to the best of the author's knowledge, this is a new result for the other equilibrium states.

\subsubsection{Formulating as a one-sided $\lambda$-decomposition}

We briefly describe the setting for Ma\~n\'e's construction, and refer to \cite{Ma78} and \cite{CFT2} for more details. We note as well that the construction in \cite{CFT2} differs slightly from both \cite{Ma78} and \cite{BFSV}, as it has a one-dimensional unstable manifold, rather than a one-dimensional stable. However, we expect that the techniques used in either paper can be applied in both settings.

Fix $d\geq 3$, and let $A\in SL(d,\ZZ)$ be such that all of its eigenvalues are simple, positive, and irrational, and exactly one lies outside the unit circle. Let $X = \mathbb{T}^d$ and set $f_A : X\to X$ to be the hyperbolic automorphism induced by $A$. Then $(X,f_A)$ has a unique measure of maximal entropy, which is Lebesgue measure. Ma\~n\'e's construction is as follows. Let $3\eps$ be an expansivity constant for $f_A$, and consider $\rho \in (0,3\eps)$. The Ma\~n\'e diffeomorphism $f_M$ is a $C^0$ perturbation of $f_A$ inside $B(p,\rho)$, where $p$ is a fixed point of $f_A$. We briefly highlight some of the main features of $f_M$ used in the decomposition arguments
\begin{itemize}
	\item $f_M|_{X\setminus B(p,\rho)} = f_A|_{X\setminus B(p,\rho)}$
	\item Outside of $B(p,\frac{\rho}{2})$, the center direction $E^c$ experiences contraction by $\theta_s < 1$, the second largest eigenvalue of $f_A$ 
	\item In $B(p,\frac{\rho}{2})$, there is a maximum amount of expansion $\theta > 1$ which can occur in the central direction, and $\theta$ can be made arbitrarily close to $1$.
\end{itemize}

The decomposition considered in \cite{CFT2} is given by taking $(x,n)$ to be a bad orbit segment if $\frac{1}{n}\sum_{i=0}^{n-1}\chi_{X\setminus B(p,\rho)}(f_M^ix) <  \eta$. However, $\chi_{X\setminus B(p,\rho)}$ is upper semicontinuous, not lower semicontinuous, which is required for a $\lambda$-decomposition. Instead, we define $\lambda = \chi_{X\setminus \overline{B(p,.9\rho)}}$, and consider the corresponding one-sided $\lambda$-decomposition with prefixes. This does not substantively affect the arguments in \cite{CFT2}, as we will outline. In particular, it just means that for $(x,n)\in\GGG(\eta)$, there is good behavior for $y\in B_n(x,.45\rho)$ as opposed to for $y\in B_n(x,\rho/2)$. The impact of this on the arguments in \cite{CFT2} is to change the constant $\frac{\rho}{2}$ to $.4\rho$ in \cite[Lemma 5.3]{CFT2}, which in turn shrinks the scale that the Bowen property is shown at by the corresponding amount. However, specification is shown at all scales, so this causes no problems. Therefore, we have that $\GGG(\eta)$ has specification for all $\eta > 0$ and any H\"older continuous potential $\ph$ has the Bowen property on $\GGG(\eta)$ for all $\eta > 0$. Furthermore, $B(\eta)$ is strictly contained in the set of bad orbit segments from \cite{CFT2}, and so the pressure estimate conditions still hold. These arguments hold for small $C^1$ perturbations of $f_M$ just as in \cite{CFT2}.

\subsubsection{Pressure estimate}

In order to apply Theorem \ref{thm: half lambda}, by the above discussion, all that we need to check is that the pressure gap holds in the product space. We will make use of Theorem \ref{thm: measure estimate}, and show that the measure of $\lambda^{-1}(0) = \overline{B(x,.9\rho)}$ is less than $\frac{1}{2}$. For this, we will appeal to the following result of \cite{BFSV}.

\begin{prop}[{\cite[Theorem 1.5]{BFSV}}]\label{BFSV prop}
Let $f : X \to X$ be an expansive homeomorphism of a compact metric space with specification, and let $\mu$ be the unique measure of maximal entropy. Let $g : X\to X$ be a continuous extension via $\pi : X\to X$ that is continuous and surjective (i.e., $f\circ\pi = \pi\circ g$). Then, if 
\begin{enumerate}
	\item $h(\pi^{-1}(\pi(x));g) = 0$ for all $x\in X$
	\item $\mu(\{\pi(x)\mid \pi^{-1}(\pi(x)) = \{x\}\}) = 1$.
\end{enumerate}
Then $(X,g)$ has a unique measure of maximal entropy, $\nu$, and $\pi_*\nu = \mu$.
\end{prop}

This is applied to Ma\~n\'e diffeomorphisms in \cite{BFSV}. We will restate the arguments using results from \cite{CFT2}, and show that they apply to $C^1$ perturbations of $f_M$ as well. Although possible, we will not treat the second condition, as it is only used to show uniqueness of $\nu$, which we already have in our setting.

The Anosov Shadowing Theorem, a proof of which can be found in \cite[Theorem 1.2.3]{Pilyugin} says that given any $g : X\to X$ sufficiently $C^0$-close to $f_A$, we can define a semiconjugacy $\pi : X\to X$ by sending $x\in X$ to the point $y$ such that $d(g^nx,f^ny) < \delta$ for some $\delta$ depending on $d^{C_0}(f_A,g)$, and bounded above by the expansivity constant for $f_A$, $3\eps$. To check the first condition, observe that if $z_1,z_2\in \pi^{-1}(\pi(x))$, then by definition of $\pi$, $d(g^nz_i, f_A^n\pi(x)) \leq \delta$. Consequently, $\Gamma_{2\delta}(x) \supset \pi^{-1}(\pi(x))$. Now, \cite[Lemma 5.8]{CFT2} shows that $g$ is entropy expansive at scale $6\eps$, and since $2\delta < 6\eps$, we have that $h(\pi^{-1}(\pi(x));g) = 0$ for all $x\in X$.

We will now use this result to show that $\nu(\lambda^{-1}(0)) < \frac{1}{2}$, where $\lambda^{-1}(0) = \overline{B(p,2\rho)}$. As $\pi_*\nu = \mu$, we see
$$\nu(\lambda^{-1}(0)) \leq \nu(\pi^{-1}(\pi(\lambda^{-1}(0)))) = \mu(\pi(\lambda^{-1}(0))) \leq \mu(\overline{B(p,2\rho + \delta)}),$$
where the last inequality comes from the fact that $d(x,\pi(x)) \leq \delta$. As $\rho$ and $\delta$ can both be made arbitrarily close to $0$ based on $d_{C^0}(g,f_A)$, and since $\mu$ is Lebesgue measure, we see that for all Ma\~n\'e diffeomorphisms and perturbations thereof sufficiently $C^0$-close to $f_A$, there is a pressure gap in the product by Theorem \ref{thm: measure estimate}. Thus, by Proposition \ref{prop: open pressure}, for all H\"older continuous potentials sufficiently close to constant, the pressure gap holds as well, and so the unique equilibrium state is $K$. In particular, it includes scalar multiples of the geometric potential $t\ph^u$, for $t$ close to $0$. We expect that the construction of Ma\~n\'e diffeomorphisms can be carried out so that $\ph^u$ is arbitrarily close to constant without increasing $d(f_A,f_M)$. In this case, the SRB measure is $K$. We leave this as a heuristic claim, as adding the details to verify this rigorously is beyond the scope of this paper.

\subsection*{Acknowledgments}
I would like to thank Dan Thompson for first suggesting this problem, and his guidance throughout this process

	\bibliographystyle{amsplain}
	\bibliography{Kpropertyreferences}

\providecommand{\bysame}{\leavevmode\hbox to3em{\hrulefill}\thinspace}
\providecommand{\MR}{\relax\ifhmode\unskip\space\fi MR }
\providecommand{\MRhref}[2]{%
  \href{http://www.ams.org/mathscinet-getitem?mr=#1}{#2}
}
\providecommand{\href}[2]{#2}
\begin{thebibliography}{10}

\bibitem{AbrRo}
L.~M. Abramov and V.~A. Rohlin, \emph{Entropy of a skew product of mappings
  with invariant measure}, Vestnik Leningrad. Univ. \textbf{17} (1962), no.~7,
  5--13.

\bibitem{Bow76}
Rufus Bowen, \emph{Some systems with unique equilibrium states}, Math. Systems
  Theory \textbf{8} (1974/75), no.~3, 193--202.

\bibitem{BCFT}
K.~Burns, V.~Climenhaga, T.~Fisher, and D.~J. Thompson, \emph{Unique
  equilibrium states for geodesic flows in nonpositive curvature}, Geom. Funct.
  Anal. \textbf{28} (2018), no.~5, 1209--1259.

\bibitem{BFSV}
J.~Buzzi, T.~Fisher, M.~Sambarino, and C.~V\'{a}squez, \emph{Maximal entropy
  measures for certain partially hyperbolic, derived from {A}nosov systems},
  Ergodic Theory Dynam. Systems \textbf{32} (2012), no.~1, 63--79.

\bibitem{CaT}
Benjamin Call and Daniel~J. Thompson, \emph{Equilibrium states for products of
  flows and the mixing properties of rank 1 geodesic flows},  (2019).

\bibitem{CKP}
Dong Chen, Lien-Yung Kao, and Kiho Park, \emph{Unique equilibrium states for
  geodesic flows over surfaces without focal points}, Nonlinearity \textbf{33}
  (2020), no.~3, 1118--1155.

\bibitem{CH96}
N.~I. Chernov and C.~Haskell, \emph{Nonuniformly hyperbolic {$K$}-systems are
  {B}ernoulli}, Ergodic Theory Dynam. Systems \textbf{16} (1996), no.~1,
  19--44.

\bibitem{CFT2}
Vaughn Climenhaga, Todd Fisher, and Daniel~J Thompson, \emph{Equilibrium states
  for man{\'e} diffeomorphisms}, Ergodic Theory and Dynamical Systems (2018),
  1--23.

\bibitem{CFT}
\bysame, \emph{Unique equilibrium states for {B}onatti--{V}iana
  diffeomorphisms}, Nonlinearity \textbf{31} (2018), no.~6, 2532.

\bibitem{ClimPes}
Vaughn Climenhaga and Yakov Pesin, \emph{Building thermodynamics for
  non-uniformly hyperbolic maps}, Arnold Math. J. \textbf{3} (2017), no.~1,
  37--82.

\bibitem{CT}
Vaughn Climenhaga and Daniel~J. Thompson, \emph{Unique equilibrium states for
  flows and homeomorphisms with non-uniform structure}, Adv. Math. \textbf{303}
  (2016), 745--799.

\bibitem{CFSS}
I.~P. Cornfeld, S.~V. Fomin, and Ya.~G. Sina\u{\i}, \emph{Ergodic theory},
  Grundlehren der Mathematischen Wissenschaften [Fundamental Principles of
  Mathematical Sciences], vol. 245, Springer-Verlag, New York, 1982, Translated
  from the Russian by A. B. Sosinski\u{\i}.

\bibitem{downarowicz}
Tomasz Downarowicz, \emph{Entropy in dynamical systems}, New Mathematical
  Monographs, vol.~18, Cambridge University Press, Cambridge, 2011.

\bibitem{DownSer}
Tomasz Downarowicz and Jacek Serafin, \emph{Fiber entropy and conditional
  variational principles in compact non-metrizable spaces}, Fund. Math.
  \textbf{172} (2002), no.~3, 217--247.

\bibitem{Furstenberg}
H.~Furstenberg, \emph{Recurrence in ergodic theory and combinatorial number
  theory}, Princeton University Press, Princeton, N.J., 1981, M. B. Porter
  Lectures.

\bibitem{KHV}
Adam Kanigowski, Federico Rodriguez~Hertz, and Kurt Vinhage, \emph{On the
  non-equivalence of the {B}ernoulli and {$K$} properties in dimension four},
  J. Mod. Dyn. \textbf{13} (2018), 221--250.

\bibitem{katokmap}
A.~Katok, \emph{Bernoulli diffeomorphisms on surfaces}, Ann. of Math. (2)
  \textbf{110} (1979), no.~3, 529--547.

\bibitem{Led78}
F.~Ledrappier, \emph{A variational principle for the topological conditional
  entropy}, Ergodic theory ({P}roc. {C}onf., {M}ath. {F}orschungsinst.,
  {O}berwolfach, 1978), Lecture Notes in Math., vol. 729, Springer, Berlin,
  1979, pp.~78--88.

\bibitem{L}
Fran\c{c}ois Ledrappier, \emph{Mesures d'\'equilibre d'entropie compl\`etement
  positive},  (1977), 251--272. Ast\'erisque, No. 50.

\bibitem{LLS}
Fran\c{c}ois Ledrappier, Yuri Lima, and Omri Sarig, \emph{Ergodic properties of
  equilibrium measures for smooth three dimensional flows}, Comment. Math.
  Helv. \textbf{91} (2016), no.~1, 65--106. \MR{3471937}

\bibitem{Ma78}
Ricardo Ma\~{n}\'{e}, \emph{Contributions to the stability conjecture},
  Topology \textbf{17} (1978), no.~4, 383--396.

\bibitem{OW98}
Donald Ornstein and Benjamin Weiss, \emph{On the {B}ernoulli nature of systems
  with some hyperbolic structure}, Ergodic Theory Dynam. Systems \textbf{18}
  (1998), no.~2, 441--456.

\bibitem{Pes}
Yakov~B. Pesin, \emph{Dimension theory in dynamical systems}, Chicago Lectures
  in Mathematics, University of Chicago Press, Chicago, IL, 1997, Contemporary
  views and applications.

\bibitem{Pilyugin}
Sergei~Yu. Pilyugin, \emph{Shadowing in dynamical systems}, Lecture Notes in
  Mathematics, vol. 1706, Springer-Verlag, Berlin, 1999.

\bibitem{PTV}
G.~Ponce, A.~Tahzibi, and R.~Var\~{a}o, \emph{On the {B}ernoulli property for
  certain partially hyperbolic diffeomorphisms}, Adv. Math. \textbf{329}
  (2018), 329--360.

\bibitem{Rat}
M.~Ratner, \emph{Anosov flows with {G}ibbs measures are also {B}ernoullian},
  Israel J. Math. \textbf{17} (1974), 380--391.

\bibitem{Rokh}
V.~A. Rohlin, \emph{Exact endomorphisms of a {L}ebesgue space}, Izv. Akad. Nauk
  SSSR Ser. Mat. \textbf{25} (1961), 499--530.

\bibitem{ShaZel}
Farruh Shahidi and Agnieszka Zelerowicz, \emph{Thermodynamics via inducing}, J.
  Stat. Phys. \textbf{175} (2019), no.~2, 351--383.

\bibitem{Wa}
Peter Walters, \emph{An introduction to ergodic theory}, Graduate Texts in
  Mathematics, vol.~79, Springer-Verlag, New York-Berlin, 1982.

\bibitem{Wa19}
TIANYU WANG, \emph{Unique equilibrium states, large deviations and {L}yapunov
  spectra for the {K}atok map}, Ergodic Theory and Dynamical Systems (2020),
  1–38.

\end{thebibliography}

\end{document}